\documentclass{birkau}

\usepackage{amsmath,amssymb,url}
\usepackage{enumitem} 
\usepackage[all]{xy} 

\numberwithin{equation}{section}

\theoremstyle{plain}
\newtheorem{theorem}{Theorem}[section]
\newtheorem{lemma}[theorem]{Lemma}

\newtheorem{corollary}[theorem]{Corollary}

\theoremstyle{definition}
\newtheorem{definition}[theorem]{Definition}

\begin{document}


\title[Fremlin's Archimedean Riesz space tensor product]{Two results on Fremlin's Archimedean Riesz space tensor product}




\author[G. Buskes]{Gerard Buskes}
\address{Department of Mathematics\\University of Mississippi\\Mississippi 38677\\USA}
\email{mmbuskes@olemiss.edu}

\corrauthor[L.P. Thorn]{Page Thorn}
\address{Department of Mathematics\\
University of Mississippi\\Mississippi 38677\\USA}
\email{lthorn@go.olemiss.edu}





\begin{abstract}
In this paper, we characterize when, for any infinite cardinal $\alpha$, the Fremlin tensor product of two Archimedean Riesz spaces (see \cite{fremlin_tensor_1972}) is Dedekind $\alpha$-complete. We also provide an example of an ideal $I$ in an Archimedean Riesz space $E$ such that the Fremlin tensor product of $I$ with itself is not an ideal in the Fremlin tensor product of $E$ with itself.
\end{abstract}

\maketitle

\section{Introduction}
The algebraic tensor product of two Archimedean Riesz spaces is rarely an Archimedean Riesz space. In \cite{fremlin_tensor_1972}, D. H. Fremlin defines the Archimedean Riesz space tensor product that is unique up to Riesz isomorphism. The algebraic tensor product is embedded as a linear subspace in Fremlin's tensor product which may be identified with the Riesz subspace generated by the algebraic tensor product. The Fremlin tensor product factors every Riesz bimorphism into the composition of the tensor map and a corresponding unique Riesz homomorphism.
\par
In 4C of \cite{fremlin_tensor_1974}, Fremlin gives an example of an Archimedean Riesz space, namely $L^2([0,1])$, that is Dedekind complete but whose Fremlin tensor product with itself lacks Dedekind completeness even when completed relative to its norm. We characterize when the Fremlin tensor product of Dedekind complete Riesz spaces is Dedekind complete.


\section{Preliminary Material}

\begin{definition}(3.1 of \cite{fremlin_tensor_1972}) Let $E$, $F$, and $G$ be Archimedean Riesz spaces. A \emph{Riesz bimorphism} is a bilinear map $T:E\times F \to G$ such that the maps
\begin{center}
$z\longmapsto T(z,y):E\to G$\\
$z\longmapsto T(x,z):F\to G$
\end{center}
are Riesz homomorphisms for all $x\in E^+$ and all $y\in F^+$.
\end{definition}

For the reader's convenience, we summarize Theorem 4.2 and Corollary 4.4 of \cite{fremlin_tensor_1972}.

\begin{theorem}\label{tensor product} (4.2 and 4.4 of \cite{fremlin_tensor_1972}) Let $E$ and $F$ be Archimedean Riesz spaces. There exists an Archimedean Riesz space $G$ and a Riesz bimorphism $\varphi\colon E\times F\to G$ with the following properties.
\begin{enumerate}[label=(\roman*)]
\item Whenever $H$ is an Archimedean Riesz space and $\psi\colon  E\times F\to H$ is a Riesz bimorphism, there is a unique Riesz homomorphism $T\colon G\to H$ such that $T\circ \varphi = \psi$.
\item If $\psi(x,y)>0$ in $H$ whenever $x>0$ in $E$ and $y>0$ in $F$, then $G$ may be identified with the Riesz subspace of $H$ generated by $\psi[E\times F]$.
\item $\varphi$ induces an embedding of the algebraic tensor product of $E$ and $F$, denoted $E\otimes F$, in $G$.
\item $E\otimes F$ is dense in $G$ in the sense that for every $w\in G$ there exist $x_0\in E$ and $y_0\in F$ such that for every $\delta >0$ there is a $v\in E\otimes F$ such that $|w-v|\leq\delta x_0\otimes y_0$.
\item If $w>0$ in $G$, then there exist $x\in E^+$ and $y\in F^+$ such that $0<x\otimes y\leq w$.
\end{enumerate}
\end{theorem}

$G$ of Theorem \ref{tensor product}  is the \emph{Archimedean Riesz space tensor product of $E$ and $F$}, denoted by $E\bar{\otimes}F$. Any Archimedean Riesz space paired with a Riesz bimorphism satisfying the universal property $(i)$ is Riesz isomorphic to $G$. The Riesz bimorphism $\otimes\colon E\times F \to E\bar{\otimes} F$ embeds the algebraic tensor product $E\otimes F$ via $\otimes (e,f) = e\otimes f$ for all $e\in E$ and $f\in F$. According to $(ii)$, $E\bar{\otimes}F$ may be identified with the Riesz space generated by elementary tensors of the form $e\otimes f$. It follows from Proposition 2.2.11 of \cite{jameson_book} that for every element $h$ of $E\bar{\otimes}F$, there exist finite subsets $I$, $J$ of $\mathbb{N}$ and $g_{ij}\in E\otimes F$ such that $$h=\sup_{i\in I}\inf_{j\in J} \{g_{ij}\},$$ where $g_{ij}=\sum_{i=1}^n e_i\otimes f_i$ for some $n\in\mathbb{N}$, $e_i\in E$, and $f_i\in F$.  


We are interested in how the tensor product performs with respect to several well-known completeness properties, given that each component has that completeness property. 

\begin{definition} Let $E$ be an Archimedean Riesz space and $\alpha$ an infinite cardinal.
\begin{enumerate}[label=(\roman*)]
\item $E$ is called \emph{Dedekind complete} if every nonempty subset of $E$ that is bounded above has a supremum (1.3 of \cite{zaanen_introduction_1997}).
\item $E$ is called \emph{Dedekind $\sigma$-complete} if every nonempty countable subset of $E$ that is bounded above has a supremum (1.3 of \cite{zaanen_introduction_1997}).
\item $E$ is called \emph{Dedekind $\alpha$-complete} if every nonempty subset $A$ of $E$ that is bounded above with cardinality no greater than $\alpha$, denoted $|A|\leq\alpha$, has a supremum (\cite{macula_1992}).
\end{enumerate}
\end{definition}

\noindent $C(X)$ is the Riesz space of continuous, real-valued functions on a topological space $X$. In a uniformly complete Riesz space, every principal ideal is Riesz isomorphic to $C(X)$ for a compact Hausdorff space $X$. We use the topological properties corresponding (via \cite{gillman_jerison}) to each type of completeness. 

\begin{definition} Let $X$ be a completely regular topological space.
\begin{enumerate}[label=(\roman*)]
\item $X$ is \emph{extremally disconnected} if every open set has open closure (1H of \cite{gillman_jerison}).
\item Every set of the form $\{x:f(x)= 0\}$ for some $f\in C(X)$ is a \emph{zero-set} of $X$; a \emph{cozero-set} is the complement of a zero-set (1.10, 1.11 of \cite{gillman_jerison}). $X$ is \emph{basically disconnected} if every cozero-set has an open closure (1H of \cite{gillman_jerison}).
\item A subset $V\subset X$ is said to be an \emph{$\alpha$-cozero set} if 
$V=\bigcup \mathcal{U},$ where $\mathcal{U}$ is a set of cozero-sets in $X$ with $|\mathcal{U}|\leq \alpha$. 
$X$ is \emph{$\alpha$-disconnected} if every $\alpha$-cozero set has open closure (\cite{macula_1992}).
\end{enumerate}
\end{definition}

\begin{theorem} Let $X$ be a completely regular topological space. Then
\begin{enumerate}[label=(\roman*)]
\item $C(X)$ is Dedekind complete if and only if $X$ is extremally disconnected (3N of \cite{gillman_jerison}).
\item $C(X)$ is Dedekind $\sigma$-complete if and only if $X$ is basically disconnected (3N of \cite{gillman_jerison}).
\item $C(X)$ is Dedekind $\alpha$-complete if and only if $X$ is $\alpha$-disconnected (\cite{stone_1949}).
\end{enumerate}
\end{theorem}

\begin{definition} Let $E$ be a Riesz space and $0<u\in E$. A sequence $\{f_n\}_{n\in\mathbb{N}}$ is said to be a \emph{$u$-uniform Cauchy sequence} if for any $\epsilon > 0$ there exists $N\in\mathbb{N}$ such that $|f_m-f_n|< \epsilon u$ for all $m$, $n\geq N$. A sequence $\{f_n\}_{n\in\mathbb{N}}$ \emph{converges $u$-uniformly} to $f$ if there exists a sequence of numbers $\epsilon_n\downarrow 0$ such that $|f-f_n|\leq \epsilon_n u$ for all $n$. $E$ is \emph{uniformly complete} if, for every $u>0$ in $E$, every $u$-uniform Cauchy sequence has a $u$-uniform limit (pgs. 49-50 of \cite{zaanen_introduction_1997}). \end{definition}

\noindent See, for instance, section 12 of \cite{zaanen_introduction_1997} for the following line-up  in $E$ (augmented with Dedekind $\alpha$-completeness) of important completeness properties.
\begin{center}
Dedekind complete $\implies$ Dedekind $\alpha$-complete $\implies$ Dedekind $\sigma$-complete $\implies$ Archimedean
\end{center}

\begin{theorem}(39.2 of \cite{luxemburg_riesz_1971}) Every Dedekind $\sigma$-complete Riesz space is uniformly complete.\end{theorem}

\begin{theorem}(see \cite{aliprantis_langford_1984}) A \emph{uniform completion} of an Archimedean Riesz space $E$ is a pair ($\tilde{E}$, $i$) consisting of a uniformly complete Riesz space $\tilde{E}$ and a Riesz homomorphism $i\colon E\to\tilde{E}$ such that for every uniformly complete Riesz space $F$ and every Riesz homomorphism $\varphi\colon E\to F$, there is a unique Riesz homomorphism $\tilde{\varphi}\colon\tilde{E}\to F$ with $\varphi=\tilde{\varphi}\circ i$.\end{theorem}

\begin{theorem}(3 of \cite{aliprantis_langford_1984})\label{uniform completion}
For an Archimedean Riesz space $E$, a uniform completion of $E$ exists and is unique up to Riesz isomorphisms. \end{theorem}

\section{The Fremlin Tensor Product of Principal Ideals}

In our study of when the Fremlin tensor product of two Dedekind complete Riesz spaces is Dedekind complete, we will need some information about when the Fremlin tensor product of two ideals is an ideal. We introduce what it means to have a strong order unit as defined in \cite{zaanen_introduction_1997} in order to prove Theorem 3.3.

\begin{definition}\label{uniformly dense} Let $E$ be an Archimedean Riesz space. A subset $V$ of $E$ is \emph{uniformly dense} in $E$ if for every $w\in E$ there exist $u\in V$ such that for every $\epsilon >0$ there is a $v\in V$ such that $|w-v|\leq\epsilon u$.\end{definition}

\begin{definition} Let $E$ be an Archimedean Riesz space and $f\in E$. The \emph{principal ideal} generated by $f$ is the set $$E_f=\{g\in E : |g|\leq \lambda |f| \text{ for some $\lambda\in\mathbb{R}^+$}\}.$$ An element $u>0$ for which $E_u=E$ is called a \emph{strong order unit} of $E$.\end{definition}

\begin{theorem}\label{dense completion} Let $E$ be a uniformly complete Riesz space with a strong order unit $u$. If $V$ is a uniformly dense Riesz subspace of $E$ containing $u$, then $E$ is a uniform completion of $V$.\end{theorem}
\begin{proof}
Let $\iota\colon V\to E$ be an embedding of $V$ into $E$. Suppose $F$ is a uniformly complete Riesz space and $\varphi\colon V\to F$ is a Riesz homomorphism. Let $x\in E$. By definition \ref{uniformly dense}, for every sequence of numbers $\epsilon_n\downarrow 0$ there exists $(x_n)_{n\in\mathbb{N}}\in V$ such that $|x-x_n|\leq\epsilon_n u$ for every $n\in\mathbb{N}$. Since $(x_n)_{n\in\mathbb{N}}$ is a $u$-uniform Cauchy sequence, for all $\epsilon > 0$ there exists $N\in\mathbb{N}$ such that $|x_n-x_m|<\epsilon u$ whenever  $n$, $m\geq N$. Then for all $\epsilon >0$, there exists $N$ such that $\forall n$, $m\geq N$, $$|\varphi(x_n)-\varphi(x_m)|=|\varphi(x_n-x_m)|=\varphi(|x_n-x_m|)<\epsilon\varphi(u).$$
Thus, $(\varphi(x_n))_{n\in\mathbb{N}}$ is a $\varphi(u)$-uniform Cauchy sequence of $F$ and converges $\varphi(u)$-uniformly in $F$. Define $\tilde{\varphi}(x)$ to be the $\varphi(u)$-uniform limit of $(\varphi(x_n))_{n\in\mathbb{N}}$ in $F$. By Theorem 10.3 (ii) of \cite{zaanen_introduction_1997}, $\tilde{\varphi}$ is a Riesz homomorphism that is uniquely determined by $\varphi$. 
Then $\varphi=\tilde{\varphi}\circ \iota$ and $(E,\iota)$ is a uniform completion of $V$.
\end{proof}

\begin{lemma}\label{dense ideal} Let $E$ and $F$ be uniformly complete Archimedean Riesz spaces. Then $E_x \bar{\otimes} F_y$ is uniformly dense in $(E\bar{\otimes}F)_{x\otimes y}$ for every $x\in E^+$ and $y\in F^+$. \end{lemma}
\begin{proof} There exist compact Hausdorff spaces $X$ and $Y$ such that $E_x$ is Riesz isomorphic to $C(X)$ and $F_y$ is Riesz isomorphic to $C(Y)$ (e.g. 4.21 and 4.29 of \cite{aliprantis_positive_2006}). Then $E_x\bar{\otimes}F_y\cong C(X)\bar{\otimes}C(Y)$. Also, $C(X)\bar{\otimes}C(Y)$ is uniformly dense in $C(X\times Y)$ where $x\in E^+$ and $y\in F^+$ correspond to the unit functions $1_X$ and $1_Y$ respectively (2.2 of \cite{allenby_labuschagne}). 
Finally, $(E\bar{\otimes}F)_{x\otimes y}$ has a unique uniform completion that is Riesz isomorphic to a Riesz subspace of $C(X\times Y)$. Since $C(X)\bar{\otimes}C(Y)$ is uniformly dense in $C(X\times Y)$, it follows that $E_x \bar{\otimes} F_y$ is uniformly dense in $(E\bar{\otimes}F)_{x\otimes y}$.
\end{proof}

\begin{theorem}\label{uc ideals} Let $E$ and $F$ be Dedekind $\alpha$-complete Riesz spaces for an infinite cardinal $\alpha$. If $E\bar{\otimes} F$ is Dedekind $\alpha$-complete, then $E_x\bar{\otimes}F_y$ is a principal ideal for every $x\in E^+$ and $y\in F^+$. In particular, $$E_x\bar{\otimes}F_y=(E\bar{\otimes}F)_{x\otimes y}.$$\end{theorem} 
\begin{proof} 
Since $E_x$, $F_y$ are in particular uniformly complete, there exist a compact Hausdorff spaces $X$ and $Y$ such that $E_x\cong C(X)$ and $F_y\cong C(Y)$.
$E_x\bar{\otimes}F_y$ contains $x\otimes y$ and is uniformly dense in ${(E\bar{\otimes}F)}_{x\otimes y}$ by Lemma \ref{dense ideal}. 
As an ideal of a Dedekind $\alpha$-complete space, ${(E\bar{\otimes}F)}_{x\otimes y}$ is Dedekind $\alpha$-complete (e.g. Theorem 12.4 of \cite{zaanen_introduction_1997}).
Thus, ${(E\bar{\otimes}F)}_{x\otimes y}$ is a uniform completion of $E_x\bar{\otimes}F_y$ by Theorem \ref{dense completion}.
\par
On the other hand, $E_x\bar{\otimes}F_y$ is Riesz isomorphic to $C(X)\bar{\otimes} C(Y)$ which has $C(X\times Y)$ as a uniform completion (e.g. Stone-Weierstrass Theorem and Theorem \ref{dense completion}).
Since the uniform completion is unique up to isomorphism, $C(X\times Y)\cong {(E\bar{\otimes}F)}_{x\otimes y}$.
Hence, $C(X\times Y)$ is Dedekind $\alpha$-complete.
Therefore, $X\times Y$ is, in particular, basically disconnected which implies $X\times Y$ is an $F$-space (e.g. \cite{comfort_hindman_negrepontis}).
The product of two infinite compact spaces cannot be an $F$-space (14Q of \cite{gillman_jerison}), so $X$ or $Y$ is finite. By Proposition 2 of \cite{hager_paper_1966}, $C(X)\bar{\otimes}C(Y)= C(X\times Y)$. Therefore, $E_x\bar{\otimes}F_y\cong{(E\bar{\otimes}F)}_{x\otimes y}$.
\end{proof}

Note that Theorem \ref{uc ideals} would be automatic if, for ideals $I\subseteq E$ and $J\subseteq F$, we had that $I\bar{\otimes}J$ is an ideal of $E\bar{\otimes}F$ to which we give a counterexample in \ref{counterexample}.

\begin{definition} A function $p:[0,\infty)\to\mathbb{R}$ is said
 to be a \emph{piecewise polynomial} if there are $n\in\mathbb{N}$ and $t_1$, $\cdots$, $t_n\in [0,\infty)$ such that $t_1<t_2<\cdots <t_n$ and $p$ is a polynomial function on $[t_n,\infty)$ and $[t_i,t_{i+1}]$ for each $i=1,\cdots, n-1$ .\end{definition}
 
$PP([0,\infty))$ is the Archimedean Riesz space of piecewise polynomials on $[0,\infty)$. For $\epsilon >0$ and $a$, $b\in[0,\infty)$, we define $$B((a,b),\epsilon)=\{(x,y)\in [0,\infty)\times [0,\infty):\sqrt{(x-a)^2+(y-b)^2}<\epsilon \}.$$
Lemma \ref{linear bounded} follows from elementary calculus.

\begin{lemma}\label{linear bounded} If $p(x)$ is a piecewise polynomial on $[0,\infty )$ and $|p(x)|\leq C x$ for some $C\in\mathbb{R}^+$, then there exists $k\in\mathbb{R}^+$ such that $p''(x)=0$ whenever $x>k$.
\end{lemma}

\begin{theorem}\label{counterexample}
Suppose $E=PP([0,\infty))$. If $p(x)=x$, then $E_p\bar{\otimes}E_p$ is not an ideal in $E\bar{\otimes}E$.
\end{theorem}

\begin{proof}
For $x$, $y\in\mathbb{R}$, let $h(x,y)=1\otimes y^2\wedge x^2\otimes 1\in E\bar{\otimes} E$.
Suppose $h\in E_p \bar{\otimes} E_p$. Then there exist finite subsets $I$, $J$ of $\mathbb{N}$ such that 
$$h(x,y)=\sup_{i\in I} \inf_{j\in J} \{g_{ij}(x,y)\}$$
for some $g_{ij}\in E_p\otimes E_p$.
For each $g_{ij}$, there exist $n\in \mathbb{N}$, $p_r(x)\in E_p$, and $q_r(y)\in E_p$ such that $$g_{ij}(x,y)=\sum_{r=1}^n p_r(x)\otimes q_r(y).$$
Since $p_r$, $q_r\in E_p$, there exists $\lambda_r\in\mathbb{R}^+$ such that 
$$|g_{ij}(x,y)|
\leq\sum_{r=1}^{n} |p_r(x)\otimes q_r(y)|
\leq \sum_{r=1}^{n} \lambda_r (x\otimes y) 
= \left(\sum_{r=1}^{n} \lambda_{r}\right) x\otimes y.$$ 
By Lemma \ref{linear bounded}, for every $i\in I$ and $j\in J$ there exists $k_{ij}\in\mathbb{R}^+$ such that 
\begin{align}
x> k_{ij} \text{ and } y> k_{ij} \implies (g_{ij})_{xx}=(g_{ij})_{yy}=0,
\end{align}
where $(g_{ij})_{xx}$, $(g_{ij})_{yy}$ are the second order partial derivatives of $g_{ij}$ with respect to $x$, $y$ respectively.

\par

Let $k=sup_{ij} \{k_{ij}\}$, $D=\{[k,\infty)\times[k,\infty)\}$, and $$S_{ij}=\{(x,y)\in D : h(x,y)=g_{ij}(x,y)\}.$$
Notice that $D= \bigcup_{i\in I, j\in J} S_{ij}$ and $int(D)=\{(k,\infty)\times (k,\infty)\}$. Then $int(\bigcup_{i\in I, j\in J} S_{ij})$ is nonempty and there exist $i_o\in I$, $j_o\in J$ such that $int(S_{i_o j_o})\neq\emptyset$.
Therefore, there exists $(c,d)\in D$ and $\epsilon > 0$ such that $B((c,d),\epsilon) \subseteq int(S_{i_o j_o})$. That is, $h(x,y)=g_{i_o j_o}(x,y)$ for every $(x,y)\in B((c,d),\epsilon)$.
It follows from (3.1) that for $(x,y)\in B((c,d),\epsilon)$,
\begin{align} 
h_{xx}(x,y)=h_{yy}(x,y)=0. 
\end{align}
On the other hand, since $h(x,y)=1\otimes y^2\wedge x^2\otimes 1$,
\begin{center}
if $a>b$, then $h_{yy}(a,b)=2\neq 0$;\\
if $a<b$, then $h_{xx}(a,b)=2\neq 0$.
\end{center}
This contradicts (3.2).
Then $h\notin E_p\bar{\otimes} E_p$. 
Since $h \leq p \otimes p$, it follows that $E_p\bar{\otimes} E_p$ is not an ideal in $E\bar{\otimes} E$.
\end{proof}

\section{Dedekind $\alpha$-complete Riesz Spaces}
We thank Dr.~Anton Schep for providing us with \cite{wada} and the slides of his talk in \cite{schep_talk} where he proved Theorem \ref{C(X) TP DC} for the Banach lattice tensor product. For an example of the Fremlin tensor product of two Dedekind complete Riesz spaces that is not Dedekind complete, see 3.5 of \cite{grobler_2022}.

\begin{theorem}\label{C(X) TP DC} Let $X$ and $Y$ be compact Hausdorff spaces and $\alpha$ an infinite cardinal. Let $C(X)$ and $C(Y)$ be Dedekind $\alpha$-complete. $C(X)\bar{\otimes}C(Y)$ is Dedekind $\alpha$-complete if and only if $X$ or $Y$ is finite. \end{theorem}
\begin{proof} 
Assume $C(X)\bar{\otimes}C(Y)$ is Dedekind $\alpha$-complete.
$C(X\times Y)$ is a uniform completion of $C(X)\bar{\otimes} C(Y)$ by Theorem \ref{dense completion}.
Since $C(X)\bar{\otimes}C(Y)$ is uniformly complete by assumption, $C(X)\bar{\otimes}C(Y)$  is Riesz isomorphic to $C(X\times Y)$.
Consequently, $X\times Y$ is $\alpha$-disconnected which implies $X\times Y$ is an $F$-space (e.g. \cite{comfort_hindman_negrepontis}).
The product of two infinite compact spaces cannot be an $F$-space (14Q of \cite{gillman_jerison}), so $X$ or $Y$ is finite.
\par
Assume that $Y$ is finite. Then $C(Y)\cong C(\{1,\cdots, n\})$ for some $n\in\mathbb{N}$.
Let $B$ be a bounded subset of $C(X)\bar{\otimes}C(Y)$ with $|B|\leq \alpha$.
For every $h\in B$, there are
$\{h_i\}_{i=1}^n \in C(X)$ such that $h_i(x)= h(x, i)$ for $x\in X$ and $i\in\{1,\cdots,n\}$.
For each $i\in\{1,\cdots,n\}$, note that $\{h_i(x)\}_{h\in B}$ is a bounded subset of $C(X)$ with cardinality no greater than $\alpha$ .
Since $C(X)$ is Dedekind $\alpha$-complete,
$$\sup_{h\in B} h(x,y) = \left\{ \begin{array}{ll}
\sup_{h\in B}h_1(x) & \text{if } y = 1 \\
\vdots & \vdots \\
\sup_{h\in B} \ h_n(x) & \text{if } y = n
\end{array}
\right.$$
exists for every $(x,y)\in X\times Y$.
Define $g_i(x)=\sup_{h\in B}h_i(x)$. The characteristic function of $\{i\}$, denoted $1_{\{i\}}$, is continuous in $C(\{1,\cdots,n\})$.
Then $$\sup_{h\in B} h = \sum_{i=1}^n g_i \otimes 1_{\{i\}}$$ is an element of $C(X)\bar{\otimes}C(\{1,\cdots,n\})$.
Thus, $C(X)\bar{\otimes}C(Y)$ is Dedekind $\alpha$-complete.
\end{proof}

Note that the second half of Theorem \ref{C(X) TP DC} follows from Hager's Proposition 2 of \cite{hager_paper_1966}: if $X$ or $Y$ is finite, then $C(X)\otimes C(Y)=C(X\times Y)$. Thus, $C(X)\bar{\otimes}C(Y)=C(X\times Y)$. Since the product of an $\alpha$-disconnected space with a finite space is $\alpha$-disconnected, $C(X\times Y)$ is Dedekind $\alpha$-complete.

\begin{definition} Let $E$ be an Archimedean Riesz space and let $I$ be a nonempty set. $c_{00}(I,E)$ is the set of all maps $f\colon I\to E$ such that $$S(f)=\{x\in I : f(x)\neq 0\}$$ is finite. We call $S(f)$ the \emph{support of $f$}. If $E=\mathbb{R}$, then $c_{00}(I,E)$ is written $c_{00}(I)$. \end{definition}

For $f$, $g\in c_{00}(I,E)$, $f\leq g$ if and only if $f(x)\leq g(x)$ in $E$ for every $x\in I$. With this pointwise ordering, $c_{00}(I,E)$ is an Archimedean Riesz space.

\begin{lemma}\label{c00 representation} Let $I$ be a nonempty set. Then $c_{00}(I)\bar{\otimes}E$ and $c_{00}(I,E)$ are Riesz isomorphic.\end{lemma}
\begin{proof} We show that $c_{00}(I,E)$ has the universal property for the Fremlin tensor product of $c_{00}(I)$ and $E$ as in Theorem \ref{tensor product}.
\par
Define ${\varphi \colon c_{00}(I) \times E \to c_{00}(I,E)}$ by $$\hspace{2cm} \varphi(f,e)(x) = f(x) e \hspace{1cm} (x\in I,\ f\in c_{00}(I),\ e\in E).$$
It is straightforward to verify that $\varphi$ is a Riesz bimorphism.
\par
Let $F$ be an Archimedean Riesz space, and suppose {$\psi \colon c_{00}(I) \times E \to F$} is a Riesz bimorphism. 
If $g \in c_{00}(I, E)$, then $g$ is uniquely represented by $g = \sum_{y\in S(g)} g(y)1_{\{y\}}$. Then $$g(x) = \sum_{y\in S(g)} \varphi(1_{\{y\}}, g(y))(x)\hspace{2cm}\text{$(x\in I)$}$$ 
and $g$ is in the Riesz space generated by the range of $\varphi$.
For $g\in c_{00}(I, E)$, define $T(g) = \sum_{y\in S(g)}\psi(1_{\{y\}},g(y))$. We have the following diagram.
 $$\xymatrix{
c_{00}(I) \times E \ar[d]_{\psi} \ar[r]^{\varphi}
& c_{00}(I, E) \ar@{.>}[ld]^{T}\\
F}$$
Let $\lambda\in\mathbb{R}$ and $f$, $g\in c_{00}(I,E)$. Then by straightforward reasoning with the supports of $f$ and $g$,
\begin{align*}
T(\lambda f+g) =& T\left(\lambda\sum_{y\in S(f)} f(y)1_{\{y\}} + \sum_{y\in S(g)} g(y)1_{\{y\}}\right)\\
=& T\left(\sum_{y\in S(f)\cup S(g)} (\lambda f+g)(y)1_{\{y\}}\right)\\
=& \sum_{y\in S(f)\cup S(g)} \psi(1_{\{y\}},(\lambda f+g)(y))\\
=& \sum_{y\in S(f)\cup S(g)} \left[\psi(1_{\{y\}},\lambda f(y)) + \psi(1_{\{y\}},g(y))\right]\\
=& \lambda \sum_{y\in S(f)} \psi(1_{\{y\}}, f(y)) + \sum_{y\in S(g)} \psi(1_{\{y\}},g(y))\\
=& \lambda T(f) + T(g).
\end{align*}
Thus, $T$ is linear.
\par
Let $f,g \in c_{00}(I, E)$ with $f \wedge g = 0$.
Then $S(f)\cap S(g)=\emptyset$. Since $\psi$ is a Riesz bimorphism,
\begin{align*}
T(f)\wedge T(g)=& T (\sum_{y\in S(f)} f(y) 1_{\{y\}})\wedge T(\sum_{y\in S(g)} g(y) 1_{\{y\}})\\
=& \sum_{y\in S(f)} \psi(1_{\{y\}},f(y))\wedge \sum_{y\in S(g)} \psi( 1_{\{y\}},g(y))\\
=& \ 0.
\end{align*}
Then $T$ is a Riesz homomorphism and $\psi=T\circ\varphi$. Consequently, $c_{00}(I, E)$ is Riesz isomorphic to ${c_{00}(I) \bar{\otimes} E}$.
\end{proof}

\begin{theorem}\label{c_00 complete} Let $I$ be a nonempty set and let $E$ be an Archimedean Riesz space. 
If $E$ is  Dedekind $\alpha$-complete for an infinite cardinal $\alpha$, then ${c_{00}(I) \bar{\otimes} E}$ is Dedekind $\alpha$-complete.
\end{theorem}
\begin{proof}
Let $B$ be a bounded subset of $c_{00}(I, E)$ such that $|B|\leq\alpha$.
Then there exists an $f\in c_{00}(I, E)$ such that $S(h )\subseteq S(f)$ for every $h\in B$.
From the Dedekind $\alpha$-completeness of $E$ it follows that $\sup_{h\in B} h(x)$ exists for each $x \in I$.
Define $$g(x) = \sup_{h\in B} h(x).$$ 
$S(g)$ is finite since $S(g)\subseteq S(f)$ and $S(f)$ is finite. Then $g \in c_{00}(I, E)$. 
By Lemma \ref{c00 representation}, ${c_{00}(I) \bar{\otimes} E} \cong c_{00}(I, E)$ is Dedekind $\alpha$-complete.
\end{proof}

\begin{corollary}\label{c_00 DC} Let $I$ be a nonempty set and let $E$ be an Archimedean Riesz space. 
If $E$ is  Dedekind complete, then ${c_{00}(I) \bar{\otimes} E}$ is Dedekind complete.
\end{corollary}
\begin{proof} Let $B$ be a bounded subset of $c_{00}(I, E)$. Then $B$ has some cardinality, say $\alpha$. In particular, $E$ is Dedekind $\alpha$-complete, so $\sup{B}$ exists in $c_{00}(I) \bar{\otimes} E$ by Theorem \ref{c_00 complete}. \end{proof}

In our concluding theorem, we combine our results on ideals and Dedekind $\alpha$-completeness to characterize exactly when the Fremlin tensor product of two Dedekind $\alpha$-complete Riesz spaces is Dedekind $\alpha$-complete.

\begin{theorem}\label{main result} Let $\alpha$ be an infinite cardinal. Suppose E and F are Dedekind $\alpha$-complete. The following are equivalent.
\begin{enumerate}
\item $E_x \bar{\otimes} F_y$ is Dedekind $\alpha$-complete for every $x \in E^+$ and $y \in F^+$.
\item $[E_x$ is finite dimensional for every $x \in E^+]$ or [$F_y$ is finite dimensional for every $y \in F^+]$.
\item $E \cong c_{00}(I)$ for a set $I \subseteq E$ or $F \cong c_{00}(J)$ for a set $J \subseteq F$.
\item $E\bar{\otimes}F \cong c_{00}(I,F)$ for a set $I \subseteq E$ or $E\bar{\otimes}F \cong c_{00}(J,E)$ for a set $J \subseteq F$.
\item $E\bar{\otimes}F$ is Dedekind $\alpha$-complete.
\end{enumerate}
\end{theorem}

\begin{proof}
$(1)\implies(2)$
Let $x\in E^+$ and $y\in F^+$. There exist compact Hausdorff topological spaces $X$, $Y$ such that $E_x\cong C(X)$ and $F_y\cong C(Y)$ (Theorems 4.21 and 4.29 of \cite{aliprantis_positive_2006}).
If $E_x \bar{\otimes} F_y$ is Dedekind $\alpha$-complete, then $X$ or $Y$ is finite by Theorem \ref{C(X) TP DC}. 
Then $E_x$ or $F_y$ is finite dimensional.
We claim that this holds either for every $x\in E^+$ or for every $y\in F^+$.
Indeed, if there exists $u\in E^+$, $v\in F^+$ such that $E_{u}$, $F_{v}$ are infinite dimensional, then $E_u\bar{\otimes}F_v$ cannot be Dedekind $\alpha$-complete by Theorem \ref{C(X) TP DC}. Thus, either $E_x$ is finite dimensional for every $x\in E^+$ or $F_y$ is finite dimensional for every $y\in F^+$.
\par
$(2)\implies(3)$ Follows from Theorem 61.4 in \cite{luxemburg_riesz_1971}.
\par 
$(3)\implies(4)$ Follows from Lemma \ref{c00 representation}.
\par
$(4)\implies(5)$ Follows from Theorem \ref{c_00 complete}.
\par 
$(5)\implies(1)$ By Theorem \ref{uc ideals}, $E_x\bar{\otimes}F_y$ is an ideal for every $x\in E^+$, $y\in F^+$. Thus, $E_x\bar{\otimes}F_y$ is Dedekind $\alpha$-complete (e.g. Theorem 12.4 of \cite{zaanen_introduction_1997}). 
\end{proof}

Theorem \ref{main result} can be generalized to a Dedekind $\beta$-complete space $E$ and a Dedekind $\gamma$-complete space $F$ for $\beta\neq\gamma$. In this case, $min\{\beta,\gamma\}$ replaces $\alpha$ in statements $(1)$ and $(5)$. We emphasize the result for $\alpha$ a countable cardinal, in which case we say Dedekind $\sigma$-complete. Since Theorem \ref{main result} is for an arbitrary infinite cardinal, we have the concluding corollary.

\begin{corollary} Suppose E and F are Dedekind complete. The following are equivalent.
\begin{enumerate}
\item $E_x \bar{\otimes} F_y$ is Dedekind complete for every $x \in E^+$ and $y \in F^+$.
\item $[E_x$ is finite dimensional for every $x \in E^+]$ or [$F_y$ is finite dimensional for every $y \in F^+]$.
\item $E \cong c_{00}(I)$ for a set $I \subseteq E$ or $F \cong c_{00}(J)$ for a set $J \subseteq F$.
\item $E\bar{\otimes}F \cong c_{00}(I,F)$ for a set $I \subseteq E$ or $E\bar{\otimes}F \cong c_{00}(J,E)$ for a set $J \subseteq F$.
\item $E\bar{\otimes}F$ is Dedekind complete.
\end{enumerate}
\end{corollary}

\section*{Acknowledgments}

The authors would like to thank Dr.~Samuel Lisi for a valuable conversation regarding Section 3.
We are indebted to Dr.~Mohamed Amine Ben Amor for spotting an error in the previous version of Lemma 3.4.
In addition, we thank the referee for numerous helpful comments that have improved this paper.
The authors declare that there is no conflict of interest.

\nocite{pinter_1971}

\end{document}